\documentclass[12pt,reqno]{amsart}

\usepackage{fullpage,amsmath,amssymb,amsfonts,mathrsfs,bm,graphicx,hyperref}

\def\bbone{{\mathchoice {\rm 1\mskip-4mu l} {\rm 1\mskip-4mu l}
{\rm 1\mskip-4.5mu l} {\rm 1\mskip-5mu l}}}

\newtheorem{theorem}{Theorem}[section]

\newtheorem{corollary}{Corollary}[section]
\newtheorem{lemma}{Lemma}[section]
\newtheorem{proposition}{Proposition}[section]

\begin{document}

\author{Abdelmalek Abdesselam}
\address{Abdelmalek Abdesselam, Department of Mathematics,
P. O. Box 400137,
University of Virginia,
Charlottesville, VA 22904-4137, USA}
\email{malek@virginia.edu}

\title{A central limit theorem for connected components of random coverings of manifolds with nilpotent fundamental groups}

\begin{abstract}
There is a well understood way of generating random coverings of a fixed manifold by sampling homomorphisms from the fundamental group of this manifold into the symmetric group. We prove a central limit theorem for the number of connected components of these random coverings when the fundamental group is nilpotent. This provides a nonabelian generalization of an earlier result by the author and Shannon Starr in the case of the torus where the fundamental group is a free abelian group of rank at least two.
Our result relies on the work of du Sautoy and Grunewald on the subgroup growth zeta functions of nilpotent groups, and on Delange's generalization of the Wiener-Ikehara Tauberian theorem.
\end{abstract}

\maketitle

\tableofcontents

\section{Introduction}\label{introsec}

Let $X$ be a fixed nonempty topological space. As is customary in the theory of the fundamental group and covering spaces, we will assume that $X$ is path-connected, locally path-connected, and semilocally simply connected. We pick a basepoint $x_0\in X$ and denote the fundamental group based at $x_0$ by $G=\pi_1(X,x_0)$. There is a well known correspondence, in fact equivalence of categories, between topological coverings $\pi:Y\rightarrow X$ and left actions $L:G\times E\rightarrow E$ of the group $G$ on some set $E$ (see~\cite[pp. 68--70]{Hatcher} and~\cite[Thm. 2.3.4]{Szamuely}). The precise correspondence as well as other useful definitions will be recalled in \S\ref{prelimsec}, for the benefit of the reader who may not be an expert in all mathematical areas relevant to the present article: probability theory, group theory, topology, and number theory. We will only consider finite covers, where $(Y,\pi)$ is $n$-sheeted, with $n$ some nonnegative integer. The corresponding set $E$ is finite with $|E|=n$, where we use the notation $|\cdot|$ for the cardinality of finite sets. The number of connected components of the cover will be denoted by $c(Y,\pi)$ and it is equal to the number of orbits for the $L$ action of $G$ on $E$ which we will similarly denote by $c(E,L)$. Such an action is the same as a group homomorphism $G\rightarrow\mathfrak{S}_{E}$ into the group $\mathfrak{S}_E$ of bijections $E\rightarrow E$ which is the symmetric group $\mathfrak{S}_n$ if $E=[n]:=\{1,2,\ldots,n\}$.
We now add the hypothesis that $G$ is finitely generated, for instance if one requires $X$ to be Hausdorff compact. Then the set of homomorphisms ${\rm Hom}(G,\mathfrak{S}_{n})$ is finite. By randomly sampling an element $\varphi\in{\rm Hom}(G,\mathfrak{S}_{n})$ we immediately get a notion of random cover $(Y,\pi)$ for $X$, as well as a notion of random manifold $Y$. This kind of model for random manifolds has received much recent attention, e.g., in the articles~\cite{BakerP,MageeN,MageeNP,MageeP,MageePvH}.
Although concerned with enumeration of coverings rather than probabilistic models, similar constructions were used in~\cite{Dijkgraaf,EskinO} (for ramified covers where $X$ is a space minus the ramification locus). For other related work on the enumeration of covers, see also~\cite{LiskovetsM} and the review~\cite{KwakLM}.

A natural problem to consider in this context, is the asymptotic study when $n\rightarrow \infty$ of the random variables given by the Betti numbers of these random manifolds given as covers of a fixed space $X$. Results of this kind are very sensitive to the nature of the group $G$, and in particular its subgroup growth properties~\cite{LubotzkyS}. For instance, \cite{BakerP} 
establish such results for groups which are central extensions of free products of finite cyclic groups. In this article, we establish a central limit theorem (CLT) for the zero-th Betti number $c(Y,\pi)$ of these random manifolds, under the hypothesis that $G$ is close to commutative, i.e., is nilpotent together with some mild technical conditions.
Our main theorem generalizes the one of~\cite{AbdesselamS} which corresponds to the $\ell$-dimensional torus $X=\mathbb{T}^{\ell}=\mathbb{R}^{\ell}/\mathbb{Z}^{\ell}$, where $G=\mathbb{Z}^{\ell}$. This article follows the same proof strategy as in~\cite{AbdesselamS}, with additional ingredients provided by the work of du Sautoy and Grunewald on subgroup growth for nilpotent groups~\cite{duSautoyG}, as well as Delange's generalization of the Wiener-Ikehara Tauberian theorem (see~\cite{PierceTBZ}).

For $n\ge 0$, and $0\le k\le n$, we denote by $A(G,n,k)$ the cardinality of the set of homomorphisms $\varphi\in{\rm Hom}(G,\mathfrak{S}_{n})$ for which $c([n],L)=k$, where $L$ is the left-action given by $\varphi$, namely, $L:G\times[n]\rightarrow [n]$, $(g,i)\mapsto L(g,i)=\sigma(i)$ with the permutation $\sigma=\varphi(g)$. With another slight abuse of notation, we will denote $c(\varphi):=c([n],L)$.
In terms of a new variable $x$, we define the polynomial
\[
\mathcal{H}_{G,n}(x)=\frac{1}{n!}\sum_{k=0}^{n}A(G,n,k) x^k\ .
\]
When $x>0$ is a fixed positive real number, we can define the probability mass function
$\mathbb{P}_{G,n,x}$ on the set ${\rm Hom}(G,\mathfrak{S}_{n})$ by
\[
\mathbb{P}_{G,n,x}(\varphi)=\frac{x^{c(\varphi)}}{n!\mathcal{H}_{G,n}(x)}\ .
\]
While the above-mentioned work on random covers typically considers uniform sampling, i.e., $x=1$, we study the more general biased measures $\mathbb{P}_{G,n,x}$ which are anlogues of the Ewens measure on random permutations (the case of the circle $X=S^1$ with $G=\mathbb{Z}$). When the homomorphism $\varphi$ is sampled according to $\mathbb{P}_{G,n,x}$, this gives rise to the random variable $K_{G,n}:=c(\varphi)=c([n],L)=c(Y,\pi)$. This is the random number of connected components or orbits for which we prove a CLT. In order to state the latter, we need to quickly recall some definitions and results from the theory of subgroup growth~\cite{LubotzkyS}, with more details provided in \S\ref{prelimsec}.

For a group $G$, and for integers $n\ge 1$, we let $a_n(G):=\{H\leq G\ |\ [G:H]=n\}$ 
which counts the number of subgroups (not necessarily normal) of index exactly $n$. When $G$ is finitely generated, we have $a_n(G)<\infty$, for all $n$.
The group $G$ is called a $\mathscr{T}$-group, when it is finitely generated, torsion-free, and nilpotent. In this case, the $a_n(G)$ grow at most polynomially in $n$, and this allows one to define the Dirichlet series, introduced in~\cite{GrunewaldSS},
\[
\zeta_G(s):=\sum_{n=1}^{\infty}\frac{a_n(G)}{n^s}\ ,
\]
the subgroup growth zeta function of the group $G$ which generalizes the Riemann zeta function corresponding to $G=\mathbb{Z}$.
This Dirichlet series converges in the half-plane ${\rm Re}(s)>\alpha_G\ge 0$, where
\[
\alpha_G:=\limsup\limits_{n\rightarrow\infty}\frac{\ln\left(\sum_{j=1}^{n}a_j(G)\right)}{\ln n}\ .
\]
When $\alpha_G>0$, this quantity is the same as the abscissa of convergence of the Dirichlet series, i.e., the infimum of the set of real numbers $\sigma$ such that the series converges for ${\rm Re}(s)>\sigma$. Since the coefficients $a_n(G)$ are obviously nonnegative, this also coincides with the abscissa of absolute convergence, which is defined similarly with the convergence requirement replaced by that of absolute convergence.
Since this does not seem to be a standard definition, let us say that a finitely generated group has at least linear subgroup growth if
\[
\exists c>0, \forall n\ge 1, a_n(G)\ge c n\ .
\]
For such a group, we must have $\alpha_G\ge 2$.
The following is a deep result by du Sautoy and Grunewald~\cite{duSautoyG}.

\begin{theorem}\label{dSGthm}
Let $G$ be an infinite $\mathscr{T}$-group, then $\alpha_G$ is a rational number, and there exists $\delta>0$ such that $\zeta_G(s)$ admits a meromorphic continuation to the domain ${\rm Re}(s)>\alpha_G-\delta$. Moreover, $\zeta_G(s)$ has a pole at $s=\alpha_G$, and no other singularity on the line ${\rm Re}(s)=\alpha_G$.
\end{theorem}

Assuming the conclusions of the theorem hold, we let $m_{G}\ge 1$ denote the order of the pole at $\alpha_G$, and let $\gamma_G$ denote the coefficient of the most singular term $\frac{1}{(s-\alpha_G)^{m_G}}$ for the Laurent expansion of $\zeta_G(s)$ at $\alpha_G$. Clearly $\gamma_G=\lim_{s\rightarrow\alpha_G}(s-\alpha_G)^{m_{G}}\zeta_{G}(s)>0$. We can now define the constant
\[
\mathcal{K}_G=\frac{\Gamma(\alpha_G)\gamma_G}{(m_G-1)!}\ ,
\]
where $\Gamma(s)$ is the Euler gamma function.
We can now state our main theorem.

\begin{theorem}\label{mainthm}
Suppose $G$ is a $\mathscr{T}$-group with at least linear subgroup growth, 
then the leading asymptotics of the mean and variance of the random variables $\mathsf{K}_{G,n}$ are given by
\begin{eqnarray*}
\mathbb{E}\mathsf{K}_{G,n}&\sim & 
\frac{1}{\alpha_G-1}\times{\alpha_G}^{-\left(\frac{m_G-1}{\alpha_G}\right)}
\times(x\mathcal{K}_G)^{\frac{1}{\alpha_G}}
\times n^{\frac{\alpha_G-1}{\alpha_G}}\times(\ln n)^{\frac{m_G-1}{\alpha_G}}
\ , \\
{\rm Var}(\mathsf{K}_{G,n}) & \sim & 
\frac{1}{\alpha_G(\alpha_G-1)}\times{\alpha_G}^{-\left(\frac{m_G-1}{\alpha_G}\right)}
\times(x\mathcal{K}_G)^{\frac{1}{\alpha_G}}
\times n^{\frac{\alpha_G-1}{\alpha_G}}\times(\ln n)^{\frac{m_G-1}{\alpha_G}}
\ . 
\end{eqnarray*}
Moreover, the normalized random variables
\[
\frac{\mathsf{K}_{G,n}-\mathbb{E}\mathsf{K}_{G,n}}{\sqrt{{\rm Var}(\mathsf{K}_{G,n})}}
\]
converge in distribution and in the sense of moments to the standard Gaussian $\mathcal{N}(0,1)$. Namely, we have
\[
\lim\limits_{n\rightarrow\infty}\mathbb{E}\left[f\left(
\frac{\mathsf{K}_{G,n}-\mathbb{E}\mathsf{K}_{G,n}}{\sqrt{{\rm Var}(\mathsf{K}_{G,n})}}\right)\right]
=\frac{1}{\sqrt{2\pi}}\int\limits_{-\infty}^{\infty}f(s)\ e^{-\frac{s^2}{2}}\ {\rm d}s\ ,
\]
for all $f$'s that are bounded continuous functions, or polynomials.
\end{theorem}

The following corollary gives a more practical version of the theorem with a sufficient condition which is easier to check, and which involves the notion of Hirsch length (see \S2 for a refresher). For a polycyclic group $G$, we will denote by $h(G)$ the Hirsch length of $G$. We will also denote by $G^{\rm ab}$ the abelianization of $G$.

\begin{corollary}\label{maincor}
The conclusion of Theorem \ref{mainthm} holds if $G$ is a $\mathscr{T}$-group with $h(G^{\rm ab})\ge 2$.
\end{corollary}

Useful running examples to help understand our result are given by the torus as well as the Heisenberg manifold.
First consider the torus $X=\mathbb{T}^{\ell}$ with $G=\mathbb{Z}^{\ell}$ with the restriction to $\ell\ge 2$. As a consequence of results by Hermite~\cite[p. 193]{Hermite} and Eisenstein~\cite[p. 355]{Eisenstein2} (see also the earlier~\cite[p. 330]{Eisenstein1} for $\ell=2$), the zeta function is then given by
\[
\zeta_{\mathbb{Z}^{\ell}}(s)=\zeta(s)\zeta(s-1)\cdots\zeta(s-\ell+1)\ ,
\]
in terms of the Riemann zeta function. See also~\cite[Ch. 15]{LubotzkyS} for five different proofs of this beautiful classical result.
The abscissa of convergence is $\alpha_{\mathbb{Z}^{\ell}}=\ell$, and
the corresponding pole is simple, i.e., $m_{\mathbb{Z}^{\ell}}=1$.
Then $\gamma_{\mathbb{Z}^{\ell}}=\zeta(2)\zeta(3)\cdots\zeta(\ell)$
and $\mathcal{K}_{\mathbb{Z}^{\ell}}=(\ell-1)!\zeta(2)\zeta(3)\cdots\zeta(\ell)$ which is the same as the constant denoted $\mathcal{K}_{\ell}$ in~\cite{AbdesselamS}.
In this case, our new theorem agrees with~\cite[Thm 1.1]{AbdesselamS}.
Note that the CLT also holds for $\ell=1$, but requires different proof techniques, and this is why we restrict to $\ell\ge 2$. For $\ell=1$, $x=1$ this is the classic theorem of Goncharov for the number of cycles of a uniformly random permutation, while for $\ell=1$, $x\neq 1$ this is a result by Hansen (see~\cite{AbdesselamS} and references therein) with sampling according to the Ewens measure.

A more interesting nonabelian example is given by the Heisenberg manifold.
Let ${\rm Heis}(\mathbb{R})$ be the group, under matrix multiplication, of matrices of the form
\[
\begin{pmatrix}
1 & x & z\\
0 & 1 & y\\
0 & 0 & 1
\end{pmatrix}
\]
where $x,y,z$ are arbitrary real numbers.
Let ${\rm Heis}(\mathbb{Z})$ be the same matrix group but with $x,y,z$ in $\mathbb{Z}$ instead. The Heisenberg manifold is the quotient $X={\rm Heis}(\mathbb{R})/{\rm Heis}(\mathbb{Z})$. See\cite[\S7]{Ramras} for a very efficient introduction to the basic topological properties of $X$, e.g., compactness, and the structure as a circle bundle over the two-dimensional torus.
The fundamental group is just $G={\rm Heis}(\mathbb{Z})$ which is nilpotent of step $2$. We also have $h({\rm Heis}(\mathbb{Z}))=3$ while $h({\rm Heis}(\mathbb{Z})^{\rm ab})=2$.
By a result of Smith~\cite[Thm. 3, p. 20]{Smith}, the corresponding zeta function is
\begin{equation}
\zeta_{{\rm Heis}(\mathbb{Z})}(s)=
\frac{\zeta(s)\zeta(s-1)\zeta(2s-2)\zeta(2s-3)}{\zeta(3s-3)}
\label{Heiszetaeq}
\ .
\end{equation}
Note that zeta functions $\zeta_G$ which are as explicit are hard to come by. For a thorough study through specific examples see~\cite{duSautoyW}, and also the review~\cite{Voll} for an introduction to this area.
Another useful, and more recent, relevant reference in the area is~\cite{Sulca}.
From the above explicit formula, we see that $\alpha_{{\rm Heis}(\mathbb{Z})}=2$ where the zeta function has a double pole, i.e., $\gamma_{{\rm Heis}(\mathbb{Z})}=2$. Moreover, near $s=2$, we have
\[
\zeta(s)_{{\rm Heis}(\mathbb{Z})}\sim
\frac{\zeta(2)^2}{2\zeta(3)(s-2)^2}\ ,
\]
hence 
\[
\gamma_{{\rm Heis}(\mathbb{Z})}=\frac{\zeta(2)^2}{2\zeta(3)}\ .
\]
Our main theorem applies to $G={\rm Heis}(\mathbb{Z})$ and shows that
$K_{{\rm Heis}(\mathbb{Z}),n}$ satisfies the CLT with mean and variance asymptotics given by
\[
\mathbb{E}\mathsf{K}_{{\rm Heis}(\mathbb{Z}),n} \sim 
\frac{\pi^2}{12\sqrt{\zeta(3)}}\times\sqrt{x}\times\sqrt{n\ln n} \ ,
\]
and
\[
{\rm Var}(\mathsf{K}_{{\rm Heis}(\mathbb{Z}),n})  \sim 
\frac{\pi^2}{24\sqrt{\zeta(3)}}\times\sqrt{x}\times\sqrt{n\ln n} \ .
\]

\section{Preliminaries}\label{prelimsec}

Given the multidisciplinary aspect of this article, we collect in this section basic material about the notions involved in this work. We cannot prove all statements and refer instead to the relevant literature, but we will try to provide all the definitions.

\subsection{Topology}

Our main reference for this subsection is~\cite[Ch. 1]{Hatcher}.
A path in space $X$ is a continuous map $\alpha:[0,1]\rightarrow X$. One has a notion of concatenation $\alpha\beta$ of two paths $\alpha$, $\beta$ for which $\alpha(1)=\beta(0)$. The concatenation $\alpha\beta$ is obtained by going over $\alpha$ first, and then going over $\beta$ as a second step.
We denote by $[\alpha]$ the homotopy class of the path $\alpha$ corresponding to continuous deformation with fixed endpoints. The fundamental group $G=\pi_1(X,x_0)$ is the set of classes $[\alpha]$ for paths which start and end at $x_0$, with the multiplication $[\alpha][\beta]:=[\alpha\beta]$. The reverse path $\overline{\alpha}$ is defined by $\overline{\alpha}(t):=\alpha(1-t)$ and gives the inverse operation $[\alpha]^{-1}=[\overline{\alpha}]$. The neutral element is $[x_0]$, the class of the constant path equal to $x_0$.

Finite coverings $\pi:Y\rightarrow X$ of the fixed space $X$ form a category denoted by $\mathsf{Cov}$. The objects are finite coverings $(Y,\pi)$. A morphism $f$ from $(Y_1,\pi_1)$ to $(Y_2,\pi_2)$
is a continuous map $f:Y_1\rightarrow Y_2$ for which $\pi_2\circ f=\pi_1$. The categorical definition of 
an isomorphism of coverings is therefore such a map $f$ which is a homeomorphism.
If $(Y_1,\pi_1)=(Y_2,\pi_2)$, such an isomorphism $f$ is called an automorphism of that cover. We will denote by ${\rm Aut}(Y,\pi)$ the automorphism group, under composition, of the cover $(Y,\pi)$. We stress that we allow coverings where $Y$ is disconnected. We even allow empty coverings where $Y=\varnothing$ and $\pi$ is the map with empty graph.

We now also introduce the category $\mathsf{Act}$ of finite sets with $G$-action. An object is a pair $(E,L)$ where $E$ is a finite set and $L:G\times E\rightarrow E$ is a left $G$-action on the set $E$. A morphism $f$ from $(E_1,L_1)$ to $(E_2,L_2)$ is a map $f:E_1\rightarrow E_2$ such that $f(L_1(g,a))=L_2(g,f(a))$ for all $g\in G$ and $a\in E_1$. Namely, such morphisms are $G$-equivariant maps.
One again has a notion of isomorphism and automorphism. The group of automorphisms of an object $(E,L)$ will be denoted by ${\rm Aut}(E,L)$.
Note that we may switch perspective by identifying a left action $L$ on $E$ with the associated group homomorphism $\varphi\in{\rm Hom}(G,\mathfrak{S}_E)$ defined in \S\ref{introsec}.

For a fixed basepoint $x_0$, we now introduce two functors relating these two categories: a functor $\mathscr{L}$ from the category of finite covers of $X$ to that of finite sets with left $G$-action, and a functor $\mathscr{C}$ going in the opposite direction. Consider an object $(Y,\pi)$ in the former category and define $E=\pi^{-1}(\{x_0\})$, the fiber over $x_0$. Let $[\alpha]\in G$, $a\in E$, and let $\widetilde{\overline{\alpha}}$ denote the lift starting at $a$ of the path $\overline{\alpha}$ in $X$. Namely, $\widetilde{\overline{\alpha}}$ is the path in $Y$ such that $\pi\circ\widetilde{\overline{\alpha}}=\overline{\alpha}$, and $\widetilde{\overline{\alpha}}(0)=a$. By the path lifting lemma, this lift exists and is unique. We let $b=\widetilde{\overline{\alpha}}(1)\in E$ which is independent of the choice of representative $\alpha$, by the homotopy lifting lemma. We then define $L([\alpha],a):=b$ which is a left $G$-action on the fiber $E$.
This construction is called the permutation representation in~\cite[pp. 68--70]{Hatcher}, and the monodromy action on the fiber in~\cite[Ch. 2]{Szamuely}.
Note that if one did not reverse the path, we would have obtained a right instead of left action. This is an unpleasant feature due to the above standard definition of multiplication in the fundamental group. This resonates with~\cite[Remark 2.3.1]{Szamuely} which advocates for the opposite convention used by Deligne.
The probabilist reader will recognize the very similar issue in the theory of Markov chains in finite state spaces, and the dilemma of chosing  to represent probability distributions by row or column vectors.
At the level of objects, the functor $\mathscr{L}$ send $(Y,\pi)$ to $\mathscr{L}(Y,\pi)=(E,L)$ just constructed.
Now consider a morphism $f:(Y_1,\pi_1)\rightarrow(Y_2,\pi_2)$. The restriction (and co-restriction) to the fiber $f|_{E_1}:E_1\rightarrow E_2$, where for $i=1,2$, $E_i=\pi_i^{-1}(\{x_0\})$. This map is $G$-equivariant is thus is a morphism $(E_1,L_1)\rightarrow(E_2,L_2)$ where $L_i$ is defined as above. Hence the functor sends $f$ to the morphism $\mathscr{L}[f]:=f|_{E_1}$.

In order to give the definition of the functor $\mathscr{C}$, we first recall the definition of the universal cover whose construction requires the property of semilocal simple connectedness mentioned at the beginning of \S\ref{introsec}. The universal cover $\widehat{\pi}:\widehat{X}\rightarrow X$ is a cover such that $\widehat{X}$ is simply connected and it is essentially unique. One way to view it concretely is to see it as the set of homotopy classes of $[\rho]$ of paths $\rho$ in $X$ with $\rho(0)=x_0$ but $\rho(1)$ unrestricted. The projection is simply $\widehat{\pi}([\rho]):=\rho(1)$ consisting in only remembering the final destination of the path.
For $[\alpha]\in G$ and $[\rho]\in\widehat{X}$, it is easy to see that $([\alpha],[\rho])\mapsto[\alpha][\rho]:=[\alpha\rho]$ gives a left action of $G$ on $\widehat{X}$. For fixed $[\alpha]$, the map $[\rho]\mapsto[\alpha][\rho]$ belongs to ${\rm Aut}(\widehat{X},\widehat{\pi})$. This action is transitive when restricted to fibers, i.e., if $[\rho_1]$ and $[\rho_2]$ project to the same endpoint in $X$, there exists an element $[\alpha]\in G$ such that $[\alpha][\rho_1]=[\rho_2]$.
This also is a free action, i.e., if $[\alpha]\in G$ and $[\rho]\in\widehat{X}$
are such that $[\alpha][\rho]=[\rho]$, then we must have $[\alpha]=[x_0]$.

We can now define the functor $\mathscr{C}$. 
Given a finite set $E$ with left-action $L$ by $G$, we define $Y:=(\widehat{X}\times E)/G$ the space of orbits for the diagonal action of $G$ on the Cartesian product of the universal cover with the set $E$. Namely, we consider pairs $([\rho_1],a_1), ([\rho_2],a_2)\in \widehat{X}\times E$ to be equivalent if there exists $[\alpha]\in G$, such that $[\alpha][\rho_1]=[\rho_2]$ and $L([\alpha],a_1)=a_2$. We denote by $[([\rho],a)]$ the equivalence class of a pair $([\rho],a)$.
We let $\pi:Y\rightarrow X$ be defined by $\pi([([\rho],a)]):=\widehat{\pi}([\rho])=\rho(1)$. The covering $(Y,\pi)=:\mathscr{C}(E,L)$ is the object in $\mathsf{Cov}$ given by the functor $\mathscr{C}$.
Now let $(E_1,L_1)$, $(E_2,L_2)$ be two objects in $\mathsf{Act}$ and let $(Y_1,\pi_1)$, $(Y_2,\pi_2)$ be the corresponding objects obtained by $\mathscr{C}$, as just explained. If $\phi:Y_1\rightarrow Y_2$ is a $G$-equivariant map, 
then $f:Y_1\rightarrow Y_2$ given by $[([\rho],a)]\mapsto [([\rho],\phi(a))]$
is well defined and is a morphism of coverings of $X$. We thus let, $\mathscr{C}[\phi]:=f$.

Given a covering $(Y,\pi)$, we define a map $\omega_{Y,\pi}:(\widehat{X}\times E)/G\rightarrow Y$, where $E$ is the fiber $\pi^{-1}(\{x_0\})$.
For $[\rho]\in\widehat{X}$ and $a\in E$, we let $\widetilde{\rho}$ be the lift to $Y$ of the path $\rho$ in $X$, with the initial condition $\widetilde{\rho}(0)=a$.
By definition, we let $\omega_{Y,\pi}([([\rho],a)]):=\widetilde{\rho}(1)$. It is not hard to see that this is an isomorphism of covers. It is also not hard to see that for all morphisms of covers $f:(Y_1,\pi_1)\rightarrow(Y_2,\pi_2)$ we have the commutative diagram property
\[
f\circ\omega_{Y_1,\pi_1}=\omega_{Y_2,\pi_2}\circ \mathscr{C}[\mathscr{L}[f]]\ .
\]
In other words, the $\omega$'s provide a natural isomorphism between $\mathscr{C}\circ\mathscr{L}$ and the identity functor of the category $\mathsf{Cov}$.

Given a finite set $E$ with left action $L$, we define a map $\lambda_{E,L}:E\rightarrow \pi^{-1}(\{x_0\})$, where $\pi$ is the projection map of the cover $Y=(\widehat{X}\times E)/G$ obtained as $\mathscr{C}(E,L)$. By definition, we let $\lambda(a):=[([x_0],a)]$, using the constant path equal to the basepoint $x_0$.
It is not hard to see that this is an isomorphism of finite sets with left $G$-action.
Moreover, if $\phi:(E_1,L_1)\rightarrow(E_2,L_2)$ is a $G$-equivariant map, we have the commutative diagram property
\[
\mathscr{L}[\mathscr{C}[\phi]]\circ\lambda_{E_1,L_1}=\lambda_{E_2,L_2}\circ
\phi\ .
\]
Hence, the $\lambda$'s provide a natural isomorphism between the identity functor of the category $\mathsf{Act}$ and the composition functor $\mathscr{L}\circ\mathscr{C}$.

The previous discussion can be summarize by the following theorem.

\begin{theorem}\label{cateqthm}
The natural transformations $\omega$, $\lambda$ are isomorphisms of functors, and they establish an equivalence of categories between $\mathsf{Cov}$, the category of finite coverings of $X$, and $\mathsf{Act}$, the category of finite sets with a left group action by $G$.
\end{theorem}

\subsection{Combinatorics}

For a cover $(Y,\pi)$ as in the previous section, we denote by ${\rm deg}(Y,\pi):=|\pi^{-1}(\{x_0\})|$ the number of sheets or degree of the cover. In the ring of formal power series $\mathbb{C}[[x,z]]$ we define the bivariate generating series
\[
\mathcal{G}_{G}(x,z):=\sum_{[Y,\pi]}\frac{x^{c(Y,\pi)}z^{{\rm deg}(Y,\pi)}}{|{\rm Aut}(Y,\pi)|}\ ,
\]
which keeps track of the number of connected components $c(Y,\pi)$ and the number of sheets ${\rm deg}(Y,\pi)$ of a cover $(Y,\pi)$.
The summation is over all isomorphism classes $[Y,\pi]$ of covers $(Y,\pi)$. Clearly, they can be collected in a countable set. The need for normalizing by the symmetry factor $|{\rm Aut}(Y,\pi)|$ is well known in physics and the combinatorics of Feynman diagrams (see~\cite{Dijkgraaf,EskinO}), in relation to Joyal's theory of combinatorial species (see~\cite{AbdesselamSLC} or the upcoming~\cite{AbdesselamA}).
A consequence of the equivalence of categories in Theorem \ref{cateqthm}
is that we can rewrite the same generating function as
\begin{equation}
\mathcal{G}_{G}(x,z)=\sum_{[E,L]}\frac{x^{c(E,L)}z^{|E|}}{|{\rm Aut}(E,L)|}\ ,
\label{leftactseries}
\end{equation}
where we sum over isomorphism classes $[E,L]$ of pairs $(E,L)$ made of a finite set $E$ and a left $G$-action $L$ on $E$.
This notion, in addition to a category, gives rise to a combinatorial species, i.e., a functor $\mathscr{F}$ from the category of finite sets with bijections into itself (see, e.g.,~\cite{AbdesselamSLC}). To an object $E$, the functor associates the set $\mathscr{F}(E)$ of all left actions $L$ by $G$ on $E$. For a morphism, i.e., bijection $\sigma:E_1\rightarrow E_2$, the functor associates the relabeling bijection $\mathscr{F}[\sigma]:\mathscr{F}(E_1)\rightarrow\mathscr{F}(E_2)$ which sends $L_1$ on $E_1$ to the action $L_2(g,a):=\sigma(L_1(g,\sigma^{-1}(a)))$ on $E_2$.
The series (\ref{leftactseries}) is an example of exponential generating series for the combinatorial species $\mathscr{F}$ of left $G$-actions.
We can also write
\[
\mathcal{G}_{G}(x,z)=\sum_{n=0}^{\infty}
\ \sum_{L\in\mathscr{F}([n])}
\frac{x^{c([n],L)}z^n}{n!}\ ,
\]
as a consequence of the orbit-stabilizer theorem, or by~\cite[Theorem 8]{AbdesselamSLC} for the natural transformation from $\mathscr{F}$ to the trivial combinatorial species.
Via the identification of left actions $L\in\mathscr{F}(E)$ with group homomorphisms $\varphi\in{\rm Hom}(G,\mathfrak{S}_E)$ mentioned in \S\ref{introsec}, we can rephrase the previous equation as
\[
\mathcal{G}_{G}(x,z)=\sum_{n=0}^{\infty}
\ \sum_{\varphi\in{\rm Hom}(G,\mathfrak{S}_n)}
\frac{x^{c(\varphi)}z^n}{n!}
=\sum_{n=0}^{\infty}\mathcal{H}_{G,n}(x)z^n\ .
\]
By a very general formula by Bantay~\cite[Eq. (1)]{Bantay}, based on the so-called exponential formula, we have
\[
\mathcal{G}_{G}(x,z)=
\exp\left(\sum_{H\leq G}\frac{x\ z^{[G:H]}}{[G:H]}\right)\ .
\]
This connects the topological enumeration of covers to the subgroup growth function, since we can also write
\begin{equation}
\mathcal{G}_{G}(x,z)=
\exp\left(x\sum_{n=1}^{\infty}\frac{a_n(G)z^n}{n}\right)\ .
\label{Bantayeq}
\end{equation}
We refer to~\cite{AbdesselamA} for more details, and in particular the derivation of the Bryan-Fulman formula, which is the special case $G=\mathbb{Z}^{\ell}$ (see~\cite{BryanF} and~\cite{AbdesselamBDV}).

\subsection{Group theory}

Our main reference for the group theory involved will be~\cite[\S6.1]{DummitF} for the elementary notions, as well as~\cite{Segal} for the more advanced treatment of polycyclic groups, and~\cite{LubotzkyS} for that of subgroup growth.
For a general group $G$, we define the commutator $[x,y]$ of two elements $x,y\in G$ by $[x,y]:=x^{-1}y^{-1}xy$. For arbitrary subsets $A,B\subseteq G$, we denote by $[A,B]$ the subgroup of $G$ generated by the subset $\{[x,y]\ |\ (x,y)\in A\times B\}$. It is easy to check that $H\trianglelefteq G$, i.e., $H$ is a normal subgroup of $G$ if and only if $[G,H]\subseteq H$. Moreover, in that case, $G/H$ is an abelian group if and only if $[G,G]\subseteq H$.
By these two remarks and an easy induction, the lower central series (of subgroups) defined by $G_0=G$, and $G_{i+1}:=[G,G_i]$ for all $i\ge 0$, is such that
\[
G=G_0\trianglerighteq G_1\trianglerighteq G_2\trianglerighteq G_3\trianglerighteq\cdots\ ,
\]
and $G_i/G_{i+1}$ is abelian for all $i\ge 0$. For $i=0$, this is the abelianization $G^{\rm ab}:=G/[G,G]$. If the series reaches the trivial group $\{1_{G}\}$ in finitely many steps, then $G$ is called a nilpotent group. If $n$ is the smallest $i$ for which $G_i=\{1_G\}$, then we say that $G$ is $n$-step nilpotent. Nilpotent of step zero means the group is trivial. Nilpotent of step one means the group is abelian. Finally, an interesting example of nilpotent group of step two is $G={\rm Heis}(\mathbb{Z})$.
Indeed, consider the matrices
\[
X=\begin{pmatrix}
1 & 1 & 0\\
0 & 1 & 0\\
0 & 0 & 1
\end{pmatrix}
\ \ ,\ \ 
Y=\begin{pmatrix}
1 & 0 & 0\\
0 & 1 & 1\\
0 & 0 & 1
\end{pmatrix}
\ \ ,\ \ 
Z=\begin{pmatrix}
1 & 0 & 1\\
0 & 1 & 0\\
0 & 0 & 1
\end{pmatrix}\ \ ,
\]
which are easily seen to generate the Heisenberg group. A quick computation shows
\[
[X,Y]=Z,\ \ [X,Z]=I\ \ ,[Y,Z]=I\ .
\]
Hence $G_0=G$, $G_1=[G,G]=\langle Z\rangle$ (the infinite cyclic group generated by $Z$), and finally $G_2=\{I\}$, which explicitly verifies the property of $G$ being 2-step nilpotent. We also have $G^{\rm ab}\simeq\mathbb{Z}^2$.

Now let $G$ be a group, and assume that there is a finite sequence of subgroups 
\[
G=H_0\trianglerighteq H_1\trianglerighteq H_2\trianglerighteq\cdots\trianglerighteq
H_n=\{1_G\}\ ,
\]
which terminates with the trivial group and such that,
$H_i/H_{i+1}$ is a cyclic group
for all $i$, $0\le i\le n-1$. Then, $G$ is called a polycyclic group. The number $h(G)$ of values of $i$ for which $H_i/H_{i+1}$ is infinite, i.e., $H_i/H_{i+1}\simeq \mathbb{Z}$ is an invariant of the group $G$. Namely, it is independent of the choice of sequence $(H_i)$. It is called the Hirsch length of $G$.
If $G$ is a finitely generated abelian group, then by the fundamental structure theorem for such groups, $G\simeq \mathbb{Z}^r\oplus {\rm Torsion}$,
where ${\rm Torsion}$ is a finite direct sum of finite cyclic groups. It is easy to see that, in this abelian situation, the Hirsch length $h(G)$ coincides with the torsion-free rank $r$.
One thus easily checks, for the Heisenberg group that $h({\rm Heis}(\mathbb{Z}))=3$, and $h({\rm Heis}(\mathbb{Z}^{\rm ab})=2$, as mentioned in \S\ref{introsec}.
The following lemma is easy to prove, and gives convenient ways to check the applicability of Corollary \ref{maincor}, and thus our main theorem \ref{mainthm}.

\begin{lemma}\label{Hirschlem}
For a polycyclic group $G$, the following are equivalent:
\begin{enumerate}
\item
$h(G^{\rm ab})\ge 2$,
\item
$\exists\ell\ge 2$, $\exists N\trianglelefteq G$ such that $G/N\simeq\mathbb{Z}^{\ell}$,
\item
$\exists N\trianglelefteq G$ such that $G/N\simeq\mathbb{Z}^{2}$,
\item
there exists a surjective group homomorphism $G\rightarrow \mathbb{Z}^{2}$.
\end{enumerate}
\end{lemma}

Now let $G$ be a group, and for $n\ge 1$,
let $\mathscr{T}_n(G)$ be the set of transitive left $G$-actions on $[n]$, i.e., the set of $L\in\mathscr{F}([n])$ such that $c([n],L)=1$.
Let $\mathscr{S}_n(G)$ be the set of subgroups of $G$ of index exactly $n$.
Consider the map $\Phi_n:\mathscr{T}_n(G)\rightarrow\mathscr{S}_n(G)$ which sends an action $L$ to $\Phi(L):=\{g\in G\ |\ L(g,1)=1\}$, i.e., the stabilizer of the preferred element $1\in[n]$. The map is surjective, because one can take the canonical action of $G$ on $G/H$ by left multiplication, where $H\in\mathscr{S}_n(G)$, and then one can transport this action to $[n]$ via a bijection which sends the left coset $H$ to $1$.
The map $(\sigma,L)\mapsto \mathscr{F}[\sigma](L)$ gives a left action of $\mathfrak{S}_n$ on $\mathscr{T}_n(G)$. It is not hard to see that
\[
\Phi(\mathscr{F}[\sigma](L))=\Phi(L)
\]
for every transitive left action $L$ on $[n]$, and for 
all bijection $\sigma:[n]\rightarrow[n]$ which satisfies $\sigma(1)=1$.
Moreover, one can show that two transitive left $G$-actions $L_1$, $L_2$ on $[n]$ will satisfy $\Phi(L_1)=\Phi(L_2)$ if and only if $L_2=\mathscr{F}[\sigma](L_1)$ for some permutation $\sigma$ which fixes the element $1$. Therefore, the map $\Phi$ is exactly $(n-1)!$-to-one, and
\[
a_n(G)=\frac{\mathscr{T}_n(G)}{(n-1)!}\ .
\]
Finiteness of $a_n(G)$ immediately follows from this, if the group $G$ is finitely generated.
For a normal subgroup $N$ of $G$, precomposition by the surjective projection $G\rightarrow G/N$ is an injective map from ${\rm Hom}(G/N,\mathfrak{S}_n)$ to ${\rm Hom}(G,\mathfrak{S}_n)$. This implies $|\mathscr{T}_n(G/N)|\le |\mathscr{T}_n(G)|$ and therefore $a_n(G/N)\le a_n(G)$, for all $n\ge 1$.

The mentioned results by Hermite and Eisenstein show that
\[
a_n(\mathbb{Z}^{\ell})=\sum_{\delta_1,\ldots,\delta_{\ell}\ge 1}
\bbone\{\delta_1\cdots\delta_{\ell}=n\}
\ \delta_{1}^{\ell-1}\delta_{2}^{\ell-2}\cdots\delta_{\ell-1}^{1}\delta_{\ell}^{0}\ ,
\]
where $\bbone\{\cdots\}$ denotes the indicator function of the enclosed condition.
In particular, $a_n(\mathbb{Z}^{2})$ is given by the sum of divisor function $\sigma(n)\ge n$. We have thus established the following lemma.

\begin{lemma} Suppose the polycyclic group $G$ satisfies any of the conditions stated in Lemma \ref{Hirschlem}, then $a_n(G)\ge n$ for all $n$, and as a consequence, the group $G$ is of at least linear subgroup growth.
\end{lemma}

\subsection{Tauberian theory}\label{Tauberiansec}

We now assume that $G$ is a $\mathscr{T}$-group with at least linear subgroup growth, so that Theorem \ref{dSGthm} applies with $\alpha_G\ge 2$.
For any real number $\beta$, define the function $(0,\infty)\rightarrow \mathbb{R}$
\[
W_{G,\beta}(u)=\sum_{\delta=1}^{\infty}\delta^{\beta} a_{\delta}(G)
\ e^{-\delta u}\ .
\]
By the term-by-term differentiation theorem, it is easy to see that $W_{G,\beta}'(u)=-W_{G,\beta+1}(u)$. These functions are thus $C^{\infty}$ on $(0,\infty)$. They take positive values and they go to $0$ when $u\rightarrow \infty$.
A key ingredient for the present work is the following leading asymptotics result for these function near zero.

\begin{proposition}\label{Wasymprop}
If $\beta>-\alpha_G$, we have, as $u\rightarrow 0^{+}$, the asymptotic equivalence
\[
W_{G,\beta}(u)\sim
\frac{\Gamma(\alpha_G+\beta)\gamma_G}{(m_G-1)!}\times u^{-(\alpha_G+\beta)}\times (-\ln u)^{m_G-1}\ .
\]
\end{proposition}

For $G=\mathbb{Z}^{\ell}$ these functions $W_{\mathbb{Z}^{\ell},\beta}$ were studied in~\cite{AbdesselamS} and~\cite{Abdesselam2025}, where they were denoted $Z_{\ell+\beta}^{[\ell]}(u)$ if $\beta\in\mathbb{Z}$. See~\cite[Eqs. (12)-(15)]{AbdesselamS} for the leading asymptotics, and~\cite[Prop. 2.1]{Abdesselam2025} for the complete asymptotics to order $O(u^{\infty})$ using the Mellin approach~\cite{FlajoletGD,Zagier} (see also~\cite{Tenenbaum}).
This requires the meromorphic continuation of $\zeta_G$ to the entire complex plane and detailed knowledge of the poles. As the example of the Heisenberg group shows, these poles may be nontrivial Riemann zeros in disguise, due to the denominator of (\ref{Heiszetaeq}).
For general $\mathscr{T}$-groups $G$, we need a less quantitative tool given by Delange's generalization of the Wiener-Ikehara Tauberian theorem which requires no information whatsoever, as to how the zeta function behaves to the left of the vertical line ${\rm Re}(s)=\alpha_G$. The same Tauberian theorem will also provide us with a proof of the following result (already in~\cite{duSautoyG}) needed for our major arc estimates.

\begin{proposition}\label{sumasymprop}
If $\beta>-\alpha_G$, we have, as the integer $N$ goes to $\infty$, the asymptotic equivalence
\[
\sum_{n=1}^{N}\delta^{\beta} a_{\delta}(G)
\sim
\frac{\gamma_G}{(m_G-1)!(\alpha_G+\beta)}\times N^{\alpha_G+\beta}\times (\ln N)^{m_G-1}\ .
\]
\end{proposition}

The Tauberian theorem we will use to derive Propositions \ref{Wasymprop} and \ref{sumasymprop} is the following.

\begin{theorem}\label{Tauberianthm}
Let $b(t)$ be a measurable function $[0,\infty)\rightarrow\mathbb{R}$ which is bounded on finite intervals, and such that
\[
G(s):=\int_0^{\infty}b(t)e^{-st}{\rm d}t
\]
converges (in Lebesgue sense, i.e., absolutely) for all $s\in\mathbb{C}$ with ${{\rm Re}(s)}>0$. Suppose also that there exists an integer $m\ge 1$ and there exists $\varepsilon_0>0$ such that
\[
\lim\limits_{t\rightarrow\infty}tb(t)
\int_{0}^{\varepsilon_0}u^{m-1}e^{-tu}{\rm d}u=1\ .
\]
Let $a(t)$ be a nondecreasing function $[0,\infty)\rightarrow\mathbb{R}$. Suppose there is some $\alpha>0$ such that
\[
f(s):=\int_0^{\infty}a(t)e^{-st}{\rm d}t
\]
converges (absolutely) for all $s\in\mathbb{C}$ with ${{\rm Re}(s)}>\alpha$.

Suppose that for some constant $c_0>0$, the function 
\[
F(s):=f(s)-c_0 G(s-\alpha)
\]
can be rewritten as
\[
F(s)=\frac{g(s)}{(s-\alpha)^{m-1}}+h(s)\ ,
\]
where $g(s)$ is a polynomial of degree at most $m-2$ if $m\ge 2$, or identically zero if $m=1$, and where $h(s)$ is a function which is continuous on ${\rm Re}(s)\ge \alpha$
and holomorphic on the domain ${\rm Re}(s)>\alpha$.

Then, the previous hypotheses imply the asymptotics
\[
a(t)=(c_0+o(1))\ e^{\alpha t}\ b(t)\ ,
\]
as $t\rightarrow\infty$.
\end{theorem}
This theorem is~\cite[Thm 5.1]{PierceTBZ} to which we refer for a pedagogical proof. It is a simpler version of a more general theorem due to Delange~\cite{Delange}.
In both applications of Theorem \ref{Tauberianthm} below, we will take
$m=m_{G}$ the order of the pole of the subgroup growth zeta function $\zeta_G(s)$, and for that $m$, we will take
\[
b(t):=\frac{t^{m-1}}{(m-1)!}\ .
\]
Thus
\[
G(s)=\frac{1}{s^m}\ ,
\]
and $\varepsilon_0=1$ works for the relevant hypothesis in Theorem \ref{Tauberianthm}.

\medskip
\noindent{\bf Proof of Proposition \ref{Wasymprop}:}
In keeping with the notations of Theorem \ref{Tauberianthm}, we define
\[
a(t):=\sum_{\delta\ge 1} \delta^{\beta}
a_{\delta}(G)\ e^{-\delta e^{-t}}\ ,
\]
which converges for $t\ge 0$ because of the polynomial growth of $a_{\delta}(G)$, and is nondecreasing.
Using Tonelli's theorem, for $\sigma:={\rm Re}(s)>0$, and performing the change of variables $w=\delta e^{-t}$,
we have
\begin{eqnarray*}
\int_0^{\infty}\left|a(t)e^{-st}\right|{\rm d}t &= &
\int_0^{\infty}a(t)e^{-\sigma t}{\rm d}t\\
 & = & 
\sum_{\delta\ge 1} \delta^{\beta}
a_{\delta}(G)
\int_0^{\infty}e^{-\delta e^{-t}}e^{-\sigma t}{\rm d}t \\
 & = & 
\sum_{\delta\ge 1} \delta^{\beta}
a_{\delta}(G) \ \delta^{-\sigma}
\int_0^{\delta}e^{-w}w^{\sigma-1}{\rm d}w \\
 & \le & \Gamma(\sigma)\zeta_G(\sigma-\beta) \\
 & < & \infty\ ,
\end{eqnarray*}
if we also require $\sigma>\alpha:=\alpha_G+\beta$, which is strictly positive by hypothesis.
Hence if ${\rm Re}(s)>\alpha$, we can redo the above calculation without the $|\cdot|$ and obtain
\begin{eqnarray*}
f(s) & := & \int_0^{\infty}a(t)e^{-st}{\rm d}t \\
 & = & \sum_{\delta\ge 1} \delta^{\beta}
a_{\delta}(G) \ \delta^{-s}
\int_0^{\delta}e^{-w}w^{s-1}{\rm d}w \\
 & = & \Gamma(s)\zeta_G(s-\beta)-\Psi(s)\ ,
\end{eqnarray*}
where
\[
\Psi(s):=\sum_{\delta\ge 1} \delta^{\beta}
a_{\delta}(G) \ \delta^{-s}
\int_{\delta}^{\infty}e^{-w}w^{s-1}{\rm d}w\ .
\]
It is easy to see that $\Psi(s)$ is an entire analytic function of $s\in\mathbb{C}$.

We let $c_0:=\Gamma(\alpha_G+\beta)\gamma_G$, so the function $F(s)$ in Theorem \ref{Tauberianthm}
becomes
\[
F(s)=\Gamma(s)\zeta_G(s-\beta)-\Psi(s)
-\frac{\Gamma(\alpha_G+\beta)\gamma_G}{(s-\alpha_G-\beta)^{m_G}}\ .
\]
By Theorem \ref{dSGthm}, the function $\zeta_G(s-\beta)$ has a pole of order $m=m_G\ge 1$ at $\alpha=\alpha_G+\beta>0$. Therefore, the function
$\Gamma(s)\zeta_G(s-\beta)$ also has a pole of the same order at the same location, and one can write the relevant Laurent expansion as
\[
\Gamma(s)\zeta_G(s-\beta)=\frac{\Gamma(\alpha_G+\beta)\gamma_G}{(s-\alpha_G-\beta)^{m_G}}
+\sum_{j=1}^{m_G-1}\frac{\nu_j}{(s-\alpha_G-\beta)^j}
+\Phi(s)\ ,
\]
for some $\Phi(s)$ which is continuous for ${\rm Re}(s)\ge\alpha_G+\beta$ and holomorphic for ${\rm Re}(s)>\alpha_G+\beta$,
and for some constants $\nu_j$. The sum over $j$ is by definition empty or identically zero if $m_G=1$.
As a result,
\[
F(s)=\frac{g(s)}{(s-\alpha_G-\beta)^{m_G-1}}+h(s)
\]
with the polynomial
\[
g(s):=\sum_{j=1}^{m_G-1}\nu_j (s-\alpha_G-\beta)^{m_G-1-j}\ ,
\]
and the function 
\[
h(s):=\Phi(s)-\Psi(s)\ ,
\]
which fulfills the requirements of Theorem \ref{Tauberianthm}.
The conclusion of latter entails
\[
a(t)=(\Gamma(\alpha_G+\beta)\gamma_G+o(1))
\ e^{(\alpha_G+\beta)t}\ b(t)\ ,
\]
as $t\rightarrow\infty$.
Let $u>0$ be related to $t$ by $u:=e^{-t}$, so that $W_{G,\beta}(u)=a(t)$, and $t\rightarrow\infty$ corresponds to $u\rightarrow 0^{+}$. The wanted asymptotics immediately follow.
\qed

\medskip
\noindent{\bf Proof of Proposition \ref{sumasymprop}:}
We will be rather brief because the steps are similar to the provious proof, and also because this example of application of Theorem \ref{Tauberianthm} has been treated in detail in~\cite{PierceTBZ} for $\beta=0$, and the shift by $\beta$ is a straightforward modification.
We again use $m=m_G$, $\alpha=\alpha_G+\beta$, and the same $b(t)$ and $G(s)$.
We define
\[
a(t):=\sum_{1\le \delta\le e^t}\delta^{\beta} a_{\delta}(G)\ .
\]
We then obtain
\[
f(s)=\frac{\zeta_G(s-\beta)}{s}\ ,
\]
when ${\rm Re}(s)>\alpha_G+\beta>0$, by hypothesis.
We let
\[
c_0:=\frac{\gamma_G}{\alpha_G+\beta}\ .
\]
It is easy to see that the hypotheses of Theorem \ref{Tauberianthm} hold, from which we deduce the asymptotics
\[
a(t)=\left(\frac{\gamma_G}{\alpha_G+\beta}+o(1)\right)\ e^{(\alpha_G+\beta)t}\ b(t)\ ,
\]
as $t\rightarrow\infty$. We then switch variables to $u=e^{t}$ and restrict to integer values $u=N$, and the wanted result is established.
\qed

For our major arc estimates, we will only need the case $\beta=1$ of Proposition \ref{sumasymprop}. As for Proposition \ref{Wasymprop}, we will need the cases $\beta\in\{-1,0,1,2\}$
explicitly collected below for ease of reference:
\begin{eqnarray}
W_{G,-1}(u) & \sim & \frac{\mathcal{K}_G}{\alpha_G-1}\times u^{-\alpha_G+1}\times (-\ln u)^{m_G-1}\ , \label{Wminuseq}\\
W_{G,0}(u) & \sim & \mathcal{K}_G\times u^{-\alpha_G}\times (-\ln u)^{m_G-1}\ ,
\label{Wzeroeq} \\
W_{G,1}(u) & \sim & \alpha_G\mathcal{K}_G\times u^{-\alpha_G-1}\times (-\ln u)^{m_G-1}\ , \label{Wpluseq} \\
W_{G,2}(u) & \sim & \alpha_G(\alpha_G+1)\mathcal{K}_G\times u^{-\alpha_G-2}\times (-\ln u)^{m_G-1}\ , \label{Wpluspluseq}
\end{eqnarray}
as $u\rightarrow 0^{+}$.

The $\beta=0$ function is particularly important for this work. The function $W_{G,0}(u)$ is a smooth decreasing bijection of $(0,\infty)$ onto itself, and we will need the following lemma which gives the asymptotics of the (compositional) inverse function $W_{G,0}^{-1}(w)$, as $w\rightarrow\infty$.

\begin{lemma}\label{calclem}
As $w\rightarrow\infty$, we have
\[
W_{G,0}^{-1}(w)\sim
\mathcal{K}_{G}^{\frac{1}{\alpha_G}}
\times\alpha_G^{-\left(\frac{m_G-1}{\alpha_G}\right)}
\times w^{-\frac{1}{\alpha_G}}\times (\ln w)^{\frac{m_G-1}{\alpha_G}}\ .
\]
\end{lemma}
The proof, which starts with (\ref{Wzeroeq}),
is left to the reader, and it is a direct application of the method used in~\cite[\S2.4]{deBruijn}. This is the analogue of~\cite[Lem 2.4]{AbdesselamS}, where $G=\mathbb{Z}^{\ell}$, $\ell\ge 2$, and which had no logarithmic corrections.

\section{Setting up the saddle point analysis}

From now on we assume the topological space $X$ satisfies the conditions stated at the beginning of \S\ref{introsec}, and that its fundamental group $G=\pi_1(X,x_0)$ satisfies the hypotheses of our main result, i.e., $G$ is a $\mathscr{T}$-group with at least linear subgroup growth. Theorem \ref{mainthm} is a consequence of the following key proposition.

\begin{proposition}\label{keyprop}
Consider the previous random variables $\mathsf{K}_{G,n}$ with distribution determined by the measure $\mathbb{P}_{G,n,x}$. 
Let $(a_n)$ be a sequence of real numbers, and let $(b_n)$ be a sequence of positive real numbers such that, as $n\rightarrow\infty$,
\begin{eqnarray*}
a_n &=& 
x W_{G,-1}\left(
W_{G,0}^{-1}\left(\frac{n}{x}\right)
\right)
+o\left(
n^{\frac{\alpha_G-1}{2\alpha_G}}(\ln n)^{\frac{m_G-1}{2\alpha_G}}
\right)\ ,  \\
b_n & = & 
\frac{1}{\sqrt{\alpha_G(\alpha_G-1)}}
\times{\alpha_G}^{-\left(\frac{m_G-1}{2\alpha_G}\right)}
\times(x\mathcal{K}_G)^{\frac{1}{2\alpha_G}}
\times n^{\frac{\alpha_G-1}{2\alpha_G}}
\times(\ln n)^{\frac{m_G-1}{2\alpha_G}}
\times (1+o(1))\ .
\end{eqnarray*}
Then, for all $s\in\mathbb{R}$, we have
\[
\lim\limits_{n\rightarrow\infty}
\ln
\mathbb{E}\left[\exp\left(s\left(\frac{\mathsf{K}_{G,n}-a_n}{b_n}\right)\right)\right]
=\frac{s^2}{2}\ .
\]
\end{proposition}

Note that, for the benefit of the reader already aquainted with~\cite{AbdesselamS}, we will follow the same proof strategy and will use very similar notation, e.g., $a_n$, $b_n$ for the needed sequences. There should be no confusion with the subgroup growth sequence $a_n(G)$ because the latter will always include the reference to the group $G$. The real variable $s$ is fixed, and should not lead to confusion with the complex argument of $\zeta_G$, used in \S\ref{Tauberiansec}, and which is no longer needed.
Since the group has polynomial subgroup growth, the series
\[
\sum_{\delta=1}^{\infty}\frac{a_{\delta}(G)z^{\delta}}{\delta}
\]
converges absolutely for $|z|<1$.
Therefore, the equality (\ref{Bantayeq}) holds not only in the sense of formal power series, but also for fixed $x>0$, as an equality of analytic functions on the disk $|z|<1$.
We can use Cauchy's formula to extract the coefficient of $z^n$ by
\[
\mathcal{H}_{G,n}(x)=\frac{1}{2i\pi}\oint_{C(r)}
z^{-n}\ \mathcal{G}_{G}(x,z)\ \frac{{\rm d}z}{z}\ ,
\]
where $C(r)$ denotes the circle of
radius $r\in(0,1)$ around the origin with counterclockwise orientation.
We will write this radius as $r=e^{-t}$ for a suitably optimized
parameter $t\in(0,\infty)$. 
We will also split $\mathcal{H}_{G,n}(x)$ as a product of a prefactor times an integral to be analyzed by the saddle point method:
\[
\mathcal{H}_{G,n}(x)=\mathcal{P}_{n}(x,t)\mathcal{J}_{n}(x,t)\ ,
\]
where
\[
\mathcal{P}_{n}(x,t):=e^{nt}\mathcal{G}_{G}(x,e^{-t})
=\exp\left(nt+xW_{G,-1}(t)\right)\ ,
\]
and
\[
\mathcal{J}_{n}(x,t):=\int_{-\pi}^{\pi}j_{n}(x,t,\theta)\ \frac{{\rm d}\theta}{2\pi}\ ,
\]
with
\[
j_{n}(x,t,\theta):=e^{-in\theta}\times\frac{\mathcal{G}_{G}(x,e^{-t+i\theta})}{\mathcal{G}_{G}(x,e^{-t})}\ .
\]
For $y,u>0$ and $-\pi<\theta<\pi$, we also define
\begin{equation}\label{qdefeq}
q_n(y,u,\theta):=in\theta-y\left(
\sum_{\delta\ge 1}\frac{(e^{-u+i\theta})^{\delta}}{\delta}a_{\delta}(G)
-\sum_{\delta\ge 1}\frac{(e^{-u})^{\delta}}{\delta}a_{\delta}(G)
\right)\ .
\end{equation}
Hence,
\[
\mathcal{J}_{n}(x,t):=\int_{-\pi}^{\pi}e^{-q_n(x,t,\theta)}\ \frac{{\rm d}\theta}{2\pi}\ .
\]
To establish Proposition \ref{keyprop}, we need to show the convergence to $\frac{s^2}{2}$ of the log-moment generating function
\[
\Psi_n(s):=
\ln
\mathbb{E}\left[\exp\left(s\left(\frac{\mathsf{K}_{G,n}-a_n}{b_n}\right)\right)\right]
\ .
\]
Again, the notation closely follows~\cite{AbdesselamS} and $\Psi_n(s)$ is unrelated to the function $\Psi(s)$ temporarily introduced for the purposes of proving Proposition \ref{sumasymprop}.
We define the two sequences
\[
t_n:=W_{G,0}^{-1}\left(\frac{n}{x}\right)\ ,
\]
and
\[
x_n:=x e^{\frac{s}{b_n}}\ .
\]
By Lemma \ref{calclem}, we immediately deduce the asypmtotics
\begin{equation}
t_n\sim (x\mathcal{K}_G)^{\frac{1}{\alpha_G}}\times
\alpha_G^{-\left(\frac{m_G-1}{\alpha_G}\right)}\times
n^{-\frac{1}{\alpha_G}}
\times (\ln n)^{\frac{m_G-1}{\alpha_G}}\ ,
\label{tasymeq}
\end{equation}
and in particular $t_n\rightarrow 0^{+}$, when
$n\rightarrow\infty$.

We also have
\[
\Psi_n(s)=-\frac{s a_n}{b_n}
+\ln \mathcal{P}_{n}(x_n,t_n)
-\ln \mathcal{P}_{n}(x,t_n)
+R_n(s)
\]
with
\begin{equation}
R_n(s):=\ln\left[
\frac{\mathcal{J}_{n}(x_n,t_n)}{\mathcal{J}_{n}(x,t_n)}
\right]\ .
\label{Rneq}
\end{equation}
Note that the constant sequence equal to $x$ can also be viewed as the $x_n$
sequence, in the special case when $s=0$.
The saddle point method \`a la Hayman requires choosing the contour or the radius $r=e^{-t}$, in order to minimize the maximum over the contour of the modulus of the integrand.
This gives the equation
\[
\frac{{\rm d}}{{\rm d}t}\left(nt+x W_{G,-1}(t)\right)=0\ ,
\]
which is precisely solved by our choice for $t_n$.
Thus the optimization is exact for the integral giving $\mathcal{H}_{G,n}(x)$, but only approximate for the integral giving $\mathcal{H}_{G,n}(x_n)$.

An easy calculation shows, as in~\cite[\S5]{AbdesselamS},
\begin{equation}
\Psi_n(s)=A_n s+ \frac{B_n s^2}{2}+x\left(e^{\frac{s}{b_n}}-1-\frac{s}{b_n}
-\frac{s^2}{2b_n^2}\right)W_{G,-1}(t_n) +R_{n}(s)\ ,
\label{Psineq}
\end{equation}
with
\begin{eqnarray}
A_n &:=& \frac{x W_{G,-1}(t_n)(t_n)-a_n}{b_n}\ ,\label{Aneq}\\
B_n & := &\frac{x W_{G,-1}(t_n)}{b_n^2} \ . \label{Bneq}
\end{eqnarray}
We will split the integral $\mathcal{J}_{n}(x_n,t_n)$ into two pieces:
a major arc region of integration
\[
\mathcal{R}_{\rm maj}:=\{\theta\in (-\pi,\pi)\ |\ |\theta|\le t_n\}\ ,
\]
and a minor arc region of integration
\[
\mathcal{R}_{\rm min}:=\{\theta\in (-\pi,\pi)\ |\ |\theta|> t_n\}\ .
\]
We will use separate estimates for these two regions, in order to determine asymptotics for the integral $\mathcal{J}_{n}(x_n,t_n)$.

\section{Major and minor arc estimates}

From the definition (\ref{qdefeq}), we deduce
\begin{eqnarray}
{\rm Re}\ q_n(x_n,t_n,\theta) & = & 
x_n \sum_{\delta=1}^{\infty}\frac{e^{-\delta t_n}}{\delta}\ a_{\delta}(G)\ \ (1-e^{i\delta\theta})\ , \label{reqeq}\\
 & = & 2x_n \sum_{\delta=1}^{\infty}
\frac{a_{\delta}(G)}{\delta}\ e^{-\delta t_n}\sin^{2}\left(
\frac{\delta \theta}{2}
\right)\ . \nonumber
\end{eqnarray}
We let 
\[
N_n:=\left\lfloor\frac{\pi}{t_n}\right\rfloor\ , 
\]
and see that for $\theta\in\mathcal{R}_{\rm maj}$ we have
\begin{eqnarray*}
{\rm Re}\ q_n(x_n,t_n,\theta) & \ge &
 2x_n \sum_{\delta=1}^{N_n}
\frac{a_{\delta}(G)}{\delta}\ e^{-\delta t_n}\sin^{2}\left(
\frac{\delta \theta}{2}
\right) \ ,\\
 & \ge & \frac{2 x_n \theta^2}{\pi^2 e^{\pi}}\sum_{\delta=1}^{N_n}
\delta\ a_{\delta}(G)\ ,
\end{eqnarray*}
where we use the convexity estimate $\sin u\ge \frac{2u}{\pi}$ for $0\le u\le\frac{\pi}{2}$, with $u=\left|\frac{\delta\theta}{2}\right|$.
This is because $1\le\delta\le N_n\le\frac{\pi}{t_n}$, and $|\theta|\le t_n$ by hypothesis. We also used $\delta t_n\le \pi$ to bound $e^{-\delta t_n}$ from below by $e^{-\pi}$.

Since $N_n\sim \frac{\pi}{t_n}$, we deduce from Proposition \ref{sumasymprop} with $\beta=1$,
\begin{eqnarray*}
\frac{2 x_n \theta^2}{\pi^2 e^{\pi}}\sum_{\delta=1}^{N_n}
\delta\ a_{\delta}(G)
& \sim &
 \frac{2 x \theta^2 \pi^{\alpha_G-1}\gamma_G}{e^{\pi}(m_G-1)!(\alpha_G+1)}
\times t_n^{-\alpha_G-1}\times(-\ln t_n)^{m_G-1}\ ,\\
 & \ge & \eta_{\rm maj}\ \theta^2\  t_n^{-\alpha_G-1}\ (-\ln t_n)^{m_G-1}\ ,
\end{eqnarray*}
for some constant $\eta_{\rm maj}>0$ (possibly dependent on $x$) for all $n$ large enough.

We now define
\[
\lambda_n:=\left(x_n W_{G,1}(t_n)\right)^{-\frac{1}{2}}\ ,
\]
which satisfies
\begin{equation}
\lambda_n\sim (x\alpha_G\mathcal{K}_G)^{-\frac{1}{2}}
\times t_n^{\frac{\alpha_G+1}{2}}\times(-\ln t_n)^{-\left(\frac{m_G-1}{2}\right)}\ ,
\label{lambdaasymeq}
\end{equation}
by (\ref{Wpluseq}).
We will do the change of variable $\theta=\lambda_n\Theta$ in the integral over the major arc region, i.e., write
\[
\int_{\mathcal{R}_{\rm maj}}e^{-q_n(x,t,\theta)}\ \frac{{\rm d}\theta}{2\pi}=
\frac{\lambda_n}{2\pi}\int_{\mathbb{R}}f_n(\Theta)\ {\rm d}\Theta\ ,
\]
where
$f_n\colon\mathbb{R}\rightarrow\mathbb{C}$ is defined by
\[
f_n(\Theta):=\bbone\{|\Theta|\le t_n\lambda_n^{-1}\}
\ e^{-q_{\ell,n}(x_n,t_n,\lambda_n\Theta)}\ ,
\]
with $n$ large enough so that $t_n<\pi$.
Note that since $\alpha_G>1$, by (\ref{lambdaasymeq}), we have $\lim_{n\rightarrow\infty} t_n\lambda_n^{-1}=\infty$, so the effective domain of integration goes to the entire real line.

From formula (\ref{qdefeq}), and because of the definition of $x_n$, $t_n$ and $\lambda_n$,  we have after the change of variables,
\[
q_n(x_n,t_n,\lambda_n\Theta)
= i\Theta\lambda_n n(1-e^{\frac{s}{b_n}})+\frac{\Theta^2}{2}+{\rm Err}\ ,
\]
where the error term is given by
\[
{\rm Err}=-x_n\sum_{\delta\ge 1}\frac{e^{-t_n\delta}}{\delta}
\ a_{\delta}(G)\ T(i\delta\lambda_n\Theta)\ ,
\]
with the notation
\[
T(v):=e^v-1-v-\frac{v^2}{2}\ .
\]
By Taylor's formula with integral remainder one easily sees that
we have the bound $T(v)\le \frac{|v|^3}{6}$, if $v$ is pure imaginary.
By applying this to $v=i\delta\lambda_n\Theta$ in the formula for the error term, we obtain
\[
|{\rm Err}|\le \frac{1}{6}x_n\Theta^3\lambda_n^3 W_{G,2}(t_n)\ .
\]
By (\ref{Wpluspluseq}) and (\ref{lambdaasymeq}), and since $\alpha_G>1$, we see that
\[
\lim\limits_{n\rightarrow \infty}\lambda_n^3 W_{G,2}(t_n)=0\ .
\]
Hence pointwise in $\Theta\in\mathbb{R}$, the error term goes to zero.

From (\ref{tasymeq}) we see that $(-\ln t_n)\sim\frac{\ln n}{\alpha_G}$. So by combining (\ref{tasymeq}) and (\ref{lambdaasymeq}), we obtain the more explicit asymptotics in term of $n$ given by
\begin{equation}
\lambda_n\sim
\frac{1}{\sqrt{\alpha_G}}
\times{\alpha_G}^{-\left(\frac{m_G-1}{2\alpha_G}\right)}
\times(x\mathcal{K}_G)^{\frac{1}{2\alpha_G}}
\times n^{-\left(\frac{\alpha_G+1}{2\alpha_G}\right)}
\times(\ln n)^{\frac{m_G-1}{2\alpha_G}}\ .
\label{lambdanasymeq}
\end{equation}
Together with the hypothesis on $b_n$ in Proposition \ref{keyprop}, this gives
\[
\lim\limits_{n\rightarrow\infty}\frac{n\lambda_n}{b_n}=\sqrt{\alpha_G-1}\ .
\]
As a result, we have pointwise in $\Theta\in\mathbb{R}$,
\[
\lim\limits_{n\rightarrow\infty} f_n(\Theta)=\exp\left(is\Theta\sqrt{\alpha_G-1}
-\frac{\Theta^2}{2}\right)\ .
\]
Recall that for $n$ large enough, and for $|\Theta|\le t_n\lambda_n^{-1}$, we have the major arc estimate
\[
q_n(x_n,t_n,\lambda_n\Theta)\ge
\eta_{\rm maj}\ \Theta^2\lambda_n^2 t_n^{-\alpha_G-1}(-\ln t_n)^{m_G-1}\ .
\]
Pick some constant $\kappa>0$ such that $\kappa<\eta_{\rm maj}\times(x\alpha_G\mathcal{K}_{G})^{-1}$, then from (\ref{lambdaasymeq}), we see
that
\[
q_n(x_n,t_n,\lambda_n\Theta)\ge \kappa\Theta^2\ .
\]
We showed that for $n$ large enough, we have for all $\Theta\in\mathbb{R}$, $|f_n(\Theta)|\le e^{-\kappa\Theta^2}$. We can now apply the dominated convergence theorem:
\[
\lim\limits_{n\rightarrow\infty}
\int_{\mathbb{R}}f_n(\Theta)\ {\rm d}\Theta
=\int_{\mathbb{R}}e^{is\Theta\sqrt{\alpha_G-1}-\frac{\Theta^2}{2}}\ {\rm d}\Theta
=\sqrt{2\pi}\ e^{-\frac{(\alpha_G-1)s^2}{2}}\ ,
\]
by the formula for the Fourier transform of a Gaussian.

We now consider angles $\theta$ in the minor arc region $\mathcal{R}_{\rm min}$.
We go back to (\ref{reqeq}) and use the at least linear growth hypothesis $a_{\delta}(G)\ge c\delta$ for some constant $c>0$. This allows us to follow~\cite[\S4.2]{AbdesselamS} almost verbatim.
We have
\[
{\rm Re}\ q_n(x_n,t_n,\theta)\ge c\ x_n\ {\rm Re}\ \rho(t_n,\theta)\ ,
\]
with
\begin{eqnarray*}
\rho(u,\theta) & := &\sum_{k=1}^{\infty}e^{-ku}\left(1-e^{ik\theta}\right)\ , \\
 & = & \frac{e^u(e^u-e^{-i\theta}-e^u e^{i\theta}+1)}{(e^{u}-1)(e^{2u}-2e^u\cos \theta+1)}\ .
\end{eqnarray*}
As explained in~\cite[\S4.2]{AbdesselamS}, we have when $\theta\in\mathcal{R}_{\rm min}$,
the lower bound and asymptotic
\[
{\rm Re}\ q_n(x_n,t_n,\theta)\ge x_n
\times\frac{e^{t_n}(e^{t_n}+1)}{e^{t_n}-1}\times
\frac{1-\cos(t_n)}{e^{2t_n}-2e^{t_n}\cos(t_n)+1}
\sim \frac{x}{2t_n}\ .
\]
In view of (\ref{tasymeq}), there exists a constant $\eta_{\rm min}>0$ (possibly depending on $x$) such that for $n$ large enough and all $\theta\in\mathcal{R}_{\rm min}$,
\[
{\rm Re}\ q_n(x_n,t_n,\theta)\ge \eta_{\rm min}\ n^{\frac{1}{\alpha_G}}(\ln n)^{-\left(\frac{m_G-1}{\alpha_G}\right)}\ .
\]
This implies the decay estimate
\[
\left|
\int_{\mathcal{R}_{\rm min}}e^{-q_{\ell,n}(x_n,t_n,\theta)}\ \frac{{\rm d}\theta}{2\pi}\right|
\le \exp\left(-\eta_{\rm min}\ n^{\frac{1}{\alpha_G}}(\ln n)^{-\left(\frac{m_G-1}{\alpha_G}\right)}\right)\ .
\]
The latter is negligible with respect to $\lambda_n$ which follows a power law times a logarithm as shown in (\ref{lambdanasymeq}). As a result, the asmyptotics of the full integral is determined by the portion due to the major arc region $\mathcal{R}_{\rm maj}$:
\begin{eqnarray}
\mathcal{J}_{n}(x,t) & \sim &
\int_{\mathcal{R}_{\rm maj}}e^{-q_{\ell,n}(x_n,t_n,\theta)}\ \frac{{\rm d}\theta}{2\pi}\ , \nonumber\\
 & \sim& \frac{\lambda_n}{2\pi}\sqrt{2\pi}\ e^{-\frac{(\alpha_G-1)s^2}{2}}
\ , \nonumber\\
 & \sim& 
\frac{e^{-\frac{(\alpha_G-1)s^2}{2}}}{\sqrt{2\pi\alpha_G}}
\times{\alpha_G}^{-\left(\frac{m_G-1}{2\alpha_G}\right)}
\times(x\mathcal{K}_G)^{\frac{1}{2\alpha_G}}
\times n^{-\left(\frac{\alpha_G+1}{2\alpha_G}\right)}
\times(\ln n)^{\frac{m_G-1}{2\alpha_G}}\ .
\label{intasymeq}
\end{eqnarray}

\section{Completion of the proof of the main theorem}

We now finish the proof of Proposition \ref{keyprop} and then that of Theorem \ref{mainthm}.
By applying (\ref{intasymeq}) twice, for the given $s\in\mathbb{R}$, and then for the particular case $s=0$ (where $x_n$ becomes the constant sequence equal to $x$), we see that the remainder (\ref{Rneq}) satisfies
\[
\lim\limits_{n\rightarrow\infty} R_n(s)=
-\ \frac{(\alpha_G-1)s^2}{2}\ .
\]
From the hypotheses on the sequences $a_n$ and $b_n$, we immediately see that $A_n$, defined in (\ref{Aneq}), is such that
\[
\lim\limits_{n\rightarrow\infty} A_n=0\ .
\]
From (\ref{Wminuseq}), (\ref{tasymeq}), and the previous remark that $(-\ln t_n)\sim\frac{\ln n}{\alpha_{G}}$, we get
\[
x W_{G,-1}(t_n)\sim
\frac{1}{(\alpha_G-1)}
\times{\alpha_G}^{-\left(\frac{m_G-1}{\alpha_G}\right)}
\times(x\mathcal{K}_G)^{\frac{1}{\alpha_G}}
\times n^{\frac{\alpha_G-1}{\alpha_G}}
\times(\ln n)^{\frac{m_G-1}{\alpha_G}}\ .
\]
The hypothesis on $b_n$ then implies that $B_n$, defined in (\ref{Bneq}), satisfies
\[
\lim\limits_{n\rightarrow\infty} B_n=\alpha_G\ .
\]
Since it asymptotically contains, in comparison to $B_n$, an extra factor $\frac{1}{b_n}$ which goes to zero, the expression
\[
x\left(e^{\frac{s}{b_n}}-1-\frac{s}{b_n}
-\frac{s^2}{2b_n^2}\right)W_{G,-1}(t_n)
\]
also goes to zero as $n\rightarrow\infty$.
As a result, the log-moment generating function given in (\ref{Psineq}) satisfies
\[
\lim\limits_{n\rightarrow\infty}\Psi_n(s)=\frac{s^2}{2}\ ,
\]
and Proposition \ref{keyprop} is established.
\qed

The proof of the main theorem now follows the same steps as in~\cite{AbdesselamS}.
We make two rounds of application of Proposition \ref{keyprop}.
We first pick
\begin{eqnarray*}
a_n &:=& 
x W_{G,-1}\left(
W_{G,0}^{-1}\left(\frac{n}{x}\right)
\right)
\ ,  \\
b_n & := & 
\frac{1}{\sqrt{\alpha_G(\alpha_G-1)}}
\times{\alpha_G}^{-\left(\frac{m_G-1}{2\alpha_G}\right)}
\times(x\mathcal{K}_G)^{\frac{1}{2\alpha_G}}
\times n^{\frac{\alpha_G-1}{2\alpha_G}}
\times(\ln n)^{\frac{m_G-1}{2\alpha_G}}\ .
\end{eqnarray*}
Then Proposition \ref{keyprop} gives the convergence of the moment generating function of the random variables $\frac{K_{G,n}-a_n}{b_n}$ to that of the standard Gaussian. By the continuity theorem of Curtiss (see~\cite[Thm 1.2]{AbdesselamS}), 
this implies the convergence of moments. From the convergence of the first two moments, we immediately conclude that the sequences given by mean and the standard deviation satisfy the hypotheses of Proposition \ref{keyprop}.
We then redefine our $a_n$ and $b_n$ sequences as
\begin{eqnarray*}
a_n & := & \mathbb{E}\mathsf{K}_{G,n}\ , \\
b_n & := & \sqrt{{\rm Var}(K_{G,n})}\ .
\end{eqnarray*}
We then apply Proposition \ref{keyprop} again, followed by the Curtiss continuity theorem, and finally conclude the proof of Theorem \ref{mainthm}.
\qed

\section{Outlook}

It would be interesting to examine the statistics of $K_{G,n}$, for $n$ large, when the function $a_n(G)$ grows faster than a polynomial. Does the variance stay bounded? Would a limit theorem towards a discrete distribution like Poisson be more appropriate than a CLT?
In~\cite{AbdesselamS}, we mentioned right-angled Artin groups $G$ as a direction of generalization away from the Abelian case. However, as soon as one removes one edge in the relevant graph, the $a_n(G)$ grows too fast~\cite{BaikPR}.
The author was thus led to the study for the even ``more commutative'' case of nilpotent groups in the present article.

We used the ad hoc techniques in~\cite{AbdesselamS}, but another possibility would to see if the CLT from~\cite{MaplesNZ} is applicable. This would require checking the Hayman log-admissibility condition which is a nontrivial task we leave to future work.

It would be interesting to investigate, if only numerically for now, log-concavity of the numbers $A(G,n,k)$ with respect to $k$, i.e., the inequalites
\begin{equation}
A(G,n,k)^2\ge A(G,n,k-1)A(G,n,k+1)\ ,
\label{logconceq}
\end{equation}
for $2\le k\le n-1$, and for a variety of groups $G$, e.g., $\mathscr{T}$-groups.
From (\ref{Bantayeq}), one easily sees that
\[
A(G,n,k)=\frac{n!}{k!}
\sum_{n_1,\ldots,n_k\ge 1}
\bbone\{n_1+\cdots+n_k=n\}\times\frac{a_{n_1}(G)\cdots a_{n_k}(G)}{n_1\cdots n_k}\ .
\]
This is the partition function of a polymer gas, in the language of statistical mechanics~\cite{GruberK}. Moreover, the relevant polymer activities only depend on the size of the polymers, and not on their shape or geometrical features. This transformation relating the $a_n(G)$ to the $A(G,n,k)$ is also called a Bell transform~\cite{Pitman}. 
The author likes to see this as a kind of non-linear
averaging of the $a_n(G)$. For $G=\mathbb{Z}^{\ell}$, the log-concavity was conjectured by Heim and Neuhauser for $\ell=2$~\cite{HeimN}.
This log-concavity conjecture was extended to all $\ell\ge 2$ by the author~\cite{AbdesselamAC}.
In the remarkable work~\cite{Starr} concerning the $\ell=2$ case, Starr was able to establish (\ref{logconceq}) in the $n\rightarrow\infty$ limit with $k$ fixed, when $k\ge 3$. However, for $\ell=2$ and $k=2$, Starr showed that (\ref{logconceq}) fails for $n$ large. Thus the conjectures by Heim-Neuhauser and the author must be amended to exclude values of $k$ that are too low. It seems the mentioned averaging must be non-linear enough in order to see log-concavity.
Finally, if log-concavity is supported by numerical evidence, it would be interesting to look at what number-theoretical consequences this log-concavity could have.

\bigskip
\noindent{\bf Acknowledgements:}
{\small
For useful correspondence or quick in person discussion,
we thank Mikhail Ershov, Ofir Gorodetsky, Bernhard Heim,
Thomas Koberda, Vyacheslav Krushkal,
Michael Magee, Markus Neuhauser,
Doron Puder, and Gerald Tenenbaum. 
We thank Ramon van Handel for insightful discussions on subgroup growth and random coverings. We thank Paul Anderson for the related ongoing work~\cite{AbdesselamA}.
Last but not least, we thank Shannon Starr for the collaboration~\cite{AbdesselamS} where many of the techniques used in this article were developed.
} 

\end{document}